\documentclass[11pt, reqno]{amsart}

\usepackage{amsfonts,amssymb,amsmath,amsthm, ulem}
\usepackage[headinclude,DIV13]{typearea}
\usepackage[utf8]{inputenc}
\usepackage{hyperref}
\usepackage{geometry}
\usepackage{graphicx, psfrag,enumerate,enumitem}

\usepackage{soul}
\usepackage{cancel}
\usepackage{color}
 \usepackage{setspace}

\newtheorem{thm}{Theorem}[section]
\newtheorem{prop}[thm]{Proposition}

\newtheorem{lemma}[thm]{Lemma}

\numberwithin{equation}{section}

\newtheorem{innercustomass}{Assumption}
\newenvironment{customass}[1]
  {\renewcommand\theinnercustomass{#1}\innercustomass}
  {\endinnercustomass}
  
\theoremstyle{definition} 
\theoremstyle{definition}

\renewcommand{\P}{\mathbb{P}}
\newcommand{\R}{\mathbb{R}}

\newcommand{\E}{\mathbb{E}}

\newcommand{\N}{\mathbb{N}}

\newcommand{\eps}{\varepsilon}

\numberwithin{equation}{section}

\begin{document}

\title[Parasite infection in a cell population]
{Parasite infection in a cell population with deaths and reinfections}

\author{Charline Smadi}
\address{Charline Smadi, Univ. Grenoble Alpes, INRAE, LESSEM, 38000 Grenoble, France
 and Univ. Grenoble Alpes, CNRS, Institut Fourier, 38000 Grenoble, France}
\email{charline.smadi@inrae.fr}

\date{}

\maketitle

\begin{abstract}
We introduce a model of parasite infection in a cell population, where cells can be infected, either at birth through maternal transmission, from a contact with the parasites reservoir, or because of the parasites released in the cell medium after the lyses of infected cells. Inside the cells and between infection events, the quantity of parasites evolves as a general non linear branching process. We study the long time behaviour of the infection.
 \end{abstract}

 \vspace{0.2in}


\bigskip


 \vspace{0.5cm}


\section*{Introduction}

The aim of this paper is to complement the study of parasite or viral infection in a cell population conducted in \cite{BT11,BPS,marguet2020long} by adding possible reinfections, either due to new contacts of the cell population with the reservoir of the pathogen, or within the cell population, due to the pathogens released in the medium after the lyses of infected cells.

Many pathogens, in particular the ones responsible for emerging infectious diseases, are transmitted both between and within a host community. The Centre for Disease Control and Prevention (2018) defines the reservoir of a pathogen as “any animal, person, plant, soil, substance or combination of any of these in which the infectious agent normally lives”. The epidemiological dynamics of infectious diseases is highly dependent on the transmission from the reservoir, and the start of an outbreak is promoted by a primary contact between the reservoir and the incidental host (i.e. an host that becomes infected but is not part of the reservoir) leading to the potential transmission of the infection to the host population. We aim at taking into account such mechanisms.

We will thus assume that a cell population may be recurrently infected by contacts of the cells with a reservoir, and that the pathogens may also be transmitted within the cell population, either when a cell dies, releasing the pathogens it contains in the cell medium, or when a cell divides into two daugther cells receiving  respectively a fraction $\Theta$ and $1-\Theta$ of the pathogens of their mother. In this work, the cell death and division rates do not depend on the quantity of pathogens in the cells. This simplification follows from the difficulty to obtain an explicit spinal process (see later). A generalization to death and division rates dependent on the quantity of parasites in the cell could be the subject of future work.

We have two specific examples in mind. First of all the protozoan parasite \textit{Toxoplasma gondii}, which is capable of infecting and replicating within nucleated mammalian and avian cells \cite{dubey1998advances,wong1993biology}. Its life cycle is divided between feline and nonfeline infections, which are correlated with sexual and asexual replications, respectively. The asexual component consists of two distinct stages of growth depending on whether the infection is in the acute or chronic phase. The tachyzoite stage defines the rapidly growing form of the parasite found during the acute phase of toxoplasmosis. Tachyzoites replicate inside a cell until they exit to infect other cells, usually after 64 to 128 parasites have accumulated per cell \cite{radke1998cell}. The second example is \textit{Plasmodium falciparum}, also a protozoan parasite. After entering a red blood cell, the merozoites grow first to a ring-shaped form and then to a larger form called a trophozoite. Trophozoites then mature to schizonts which divide several times to produce new merozoites. The infected red blood cell eventually bursts, allowing the new merozoites to travel within the bloodstream to infect new red blood cells. Most merozoites continue this replicative cycle. However some merozoites upon infecting red blood cells differentiate into male or female sexual forms. Here again we will only be interested in the asexual part of the pathogen cycle.

The reservoir of \textit{Toxoplasma gondii} consists essentially in contaminated meat, unpasteurized goat’s milk, fresh plant products, and water \cite{hill2016toxoplasma}. In the case of \textit{Plasmodium falcipaum}, human infection is due to bites by infected mosquitoes.\\

The dynamics of the quantity of parasites in a cell will be given by a Stochastic Differential Equation (SDE) with a diffusive term and positive jumps.
Then, 
after an exponential time, the cell dies, or divides and shares its parasites between its two daughter cells.
We are interested in the long time behaviour of the parasite infection in the cell population.
We refer to \cite{marguet2020long} for references on biological studies and on previous mathematical works on the class of processes under study.
In particular, references \cite{Kimmel,bansaye2008,bansaye2009,alsmeyer2013, alsmeyer2016,meleard2013evolutive, osorio2020level} focused on branching within branching to study host-parasites dynamics.
Similarly as in \cite{georgii2003supercritical,hardy2009spine,bansaye2011limit,marguet2016uniform,cloez2017limit,marguet2020long}, 
the proof strategy consists in introducing a spinal process, and investigating the long time behaviour of this process that corresponds to the trait of a uniformly sampled individual in the population.
Via a Many-to-One formula (see Proposition \ref{pro_MtO}), we may deduce properties on the 
long time behaviour of the process at the population level.
Even if the cell division rate does not depend on the quantity of parasites, the spinal process may be non-homogeneous, because of the reinfection mechanism. Indeed the  quantity of parasites released in the medium due to the cells lyses depends on the quantity of parasites in the cell population.\\

In the sequel $\N:=\{0,1,2,...\}$ will denote the set of nonnegative integers, $\R_+:=[0,\infty)$ the real line, $\bar{\R}_+ = \R_+ \cup \{\infty\}$
and $\R_+^*:=(0,\infty)$.
We will denote by $\mathcal{C}_b^2(A)$ the set of twice continuously differentiable bounded functions on a set $A$. Finally, for any stochastic 
process $X$ on $\bar{\mathbb{R}}_+$ or $Z$ on the set of point measures on $\bar{\mathbb{R}}_+$, we will denote by 
$\mathbb{E}_x\left[f(X_t)\right]=\E\left[f(X_t)\big|X_0 = x\right]$, $\mathbb{E}\left[f(X_t)\right]=\mathbb{E}_0\left[f(X_t)\right]$ and $\mathbb{E}_{\delta_x}\left[f(Z_t)\right]=\E\left[f(Z_t)\big|Z_0 = \delta_x\right]$.

\section{Definition of the population process} \label{section_model}

\subsection{Parasites dynamics in a cell}

Each cell contains parasites whose quantity evolves as a diffusion with positive jumps. More precisely, we consider the SDE
\begin{align}  \label{X_sans_sauts} \mathfrak{X}_t =x + \int_0^t g(\mathfrak{X}_s)ds
+\int_0^t\sqrt{2\sigma^2 (\mathfrak{X}_s)}dB_s +
\int_0^t\int_0^{p(\mathfrak{X}_{s^-})}\int_{\mathbb{R}_+}z\widetilde{Q}(ds,dx,dz),
\end{align}
where $x$ is nonnegative, $g$, $\sigma \geq 0$ and $p\geq0$ are real functions on $\mathbb{R}_+$, $B$ is a standard Brownian motion,  
$\widetilde{Q}$ is a compensated Poisson point measure with intensity 
$ds\otimes dx\otimes \pi(dz)$, and $\pi$ is a nonnegative measure on $\mathbb{R}_+$.
Finally, $B$ and $Q$ are independent.

Under our conditions (see later), the SDE \eqref{X_sans_sauts} has a unique non-negative strong solution.
In this case, it is a Markov process with infinitesimal generator $\mathcal{G}$, satisfying for all $f\in C_b^2(\bar{\mathbb{R}}_+)$,
\begin{align} \label{def_gene}
\mathcal{G}f(x) = g(x)f'(x)+\sigma^2(x)f''(x)+p(x)\int_{\mathbb{R}_+}\left(f(x+z)-f(x)-zf'(x)\right)\pi(dz) ,
\end{align}
and $0$ and $+ \infty$ are two absorbing states.
Following \cite{marguet2016uniform}, we denote by $(\Phi(x,s,t),s\leq t)$ the corresponding stochastic flow {\it i.e.} 
the unique strong solution to \eqref{X_sans_sauts} satisfying $\mathfrak{X}_s = x$ and the dynamics of the trait between division events is well-defined. 

\subsection{Reinfection}

New parasites can enter cells through two distinct mechanisms: 
\begin{itemize}
\item either the population comes into contact with the reservoir of the parasite and each cell is infected at a rate $\lambda(.)$ by a quantity of parasites $\mathfrak{i}$ that follows a law $\mathfrak{I}$. This rate may depend on the quantity of parasites in the cell. An increasing rate with the quantity of parasites could indicate that the cell's defences are weakened or that the parasite is altering the properties of its host to facilitate infection by new parasites. A decreasing rate could mean that the immune system has detected the parasite and is taking measures to limit the infection;
\item or the cells are infected by parasites released by other cells at death. The rate of infection $r(.)$ then depends on the amount present in each cell. Ideally, it would be convenient to be able to follow the parasites released by a given cell at death, but the equations obtained would not be closed and the dynamics of the infection could not be studied. A simplifying hypothesis is therefore made by assuming that the infection rate is a function of the mean quantity of parasites in the population, which seems reasonable as the cell death rate does not depend on the quantity of parasites in the cell.
The quantity $\mathfrak{p}$ of parasites entering a cell during such an infection event follows a law $\mathfrak{P}$.
\end{itemize}

\subsection{Cell division and death}

Each cell divides at rate $b$ and is replaced by two 
cells with quantity of parasites at birth given by $\Theta x$ and $(1-\Theta)x$. Here $\Theta$ is a nonnegative random variable on $(0,1)$ 
with associated symmetric distribution $\kappa(d\theta)$ satisfying $\E[ \ln (1/\Theta)]<\infty$.
Finally, each cell dies at rate $d$.

\subsection{Existence and uniqueness}

We use the classical Ulam-Harris-Neveu notation to identify each individual. Let us denote by
\begin{equation*} \label{ulam_not} \mathcal{U}:=\bigcup_{n\in\mathbb{N}}\left\{0,1\right\}^{n}
\end{equation*}
the set of possible labels, $\mathcal{M}_P({\mathbb{R}}_+)$ the set of point measures on ${\mathbb{R}}_+$, and 
$\mathbb{D}(\mathbb{R}_+,\mathcal{M}_P({\mathbb{R}}_+))$, the set of c\`adl\`ag measure-valued processes. 
For any $Z\in \mathbb{D}(\mathbb{R}_+,\mathcal{M}_P({\mathbb{R}}_+))$, $t \geq 0$, we write 
\begin{equation} \label{Ztdirac}
Z_t = \sum_{u\in V_t}\delta_{X_t^u},
\end{equation}
where $V_t\subset\mathcal{U}$ denotes the set of individuals alive at time $t$ and $X_t^u$ the trait at time $t$ of the individual $u$. 
Let $E = \mathcal{U}\times(0,1)\times \R_+$, $E_{r1} = E_{r2} = \mathcal{U}\times \R_+ \times \mathbb{R}_+$  and $M(ds,du,d\theta,dz)$, $M_{r1}(ds,du,d\mathfrak{i},dz)$, $M_{r2}(ds,du,d\mathfrak{p},dz)$ be independent Poisson point measures on $\mathbb{R}_+\times E$, $\mathbb{R}_+\times E_{r1}$, $\mathbb{R}_+\times E_{r2}$ with respective intensity 
$ds\times n(du)\times \kappa(d\theta)$, $ds\times n(du)\times \mathfrak{I}(d\mathfrak{i})\times dz$, $ds\times n(du)\times \mathfrak{P}(d\mathfrak{p})\times dz$, where $n(du)$ denotes the counting measure on $\mathcal{U}$. 
Let \\ $\left(\Phi^u(x,s,t),u\in\mathcal{U},x\in\bar{\mathbb{\R}}_+, s\leq t\right)$ be a family of independent stochastic flows satisfying \eqref{X_sans_sauts} 
describing the individual-based dynamics.
We assume that $M$, $M_{r1}$, $M_{r2}$ and $\left(\Phi^u,u\in\mathcal{U}\right)$ are independent. We denote by $(\mathcal{F}_t, t \geq 0)$ the filtration generated by the Poisson point measures $M$, $M_{r1}$ and $M_{r2}$ and 
the family of stochastic processes $(\Phi^u(x,s, t), u\in\mathcal{U},x\in\bar{\mathbb{R}}_+,s\leq t)$ up to time $t$. \\

We now introduce assumptions to ensure the existence and strong uniqueness of the process. 
Points $i)$ to $iii)$ of Assumption {\bf{EU}} 
(for Existence and Uniqueness) ensure that 
the dynamics in a cell line is well defined between the reinfection times
(as the unique nonnegative strong solution to the SDE \eqref{X_sans_sauts} up to explosion, and infinite value of the quantity of parasites after explosion); point $iv)$ ensures that a cell has a higher probability to be infected when there are more parasites in the cell medium, and that if there is no parasite in the cell, there is no reinfection.

\begin{customass}{\bf{EU}}\label{ass_A}
\begin{enumerate}[label=\roman*)]
\item The function $p$ is locally Lipschitz on $\R_+$, non-decreasing and $p(0)= 0$. The function $g$ is continuous on $\R_+$, $g(0)=0$ and for any $n \in \N$ there exists a finite constant $B_n$ such that for any $0 \leq x \leq y \leq n$
\begin{align*} |g(y)-g(x)|
\leq B_n \phi(y-x),
\ 
\text{ where }\ 
\phi(x) = \left\lbrace\begin{array}{ll}
x \left(1-\ln x\right) & \textrm{if } x\leq 1,\\
1 & \textrm{if } x>1.\\
\end{array}\right.
\end{align*}
\item The function $\sigma$ is H\"older continuous with index $1/2$ on compact sets and $\sigma(0)=0$. 
\item The measure $\pi$ satisfies
$ \int_0^\infty (z\wedge z^2)\pi(dz)<\infty. $
\item $r$ and $\lambda$ are continuous, $r$ is non-decreasing and bounded, $ \int_0^\infty \mathfrak{i}\mathfrak{I}(d\mathfrak{i})+\int_0^\infty \mathfrak{p}\mathfrak{P}(d\mathfrak{p})<\infty$, and $r(0)=0$.
\end{enumerate}
\end{customass}

Recall the definition of $\mathcal{G}$ in \eqref{def_gene}.
Then the structured population process may be defined as the strong solution to a SDE. 

\begin{prop} \label{pro_exi_uni}
Under Assumption \ref{ass_A} there exists a strongly unique $\mathcal{F}_t$-adapted 
c\`adl\`ag process $(Z_t,t\geq 0)$ taking values in $\mathcal{M}_P(\bar{\mathbb{R}}_+)$ such that for all $f\in C_b^2(\bar{\mathbb{R}}_+)$ and $x_0,t\geq 0$, 
\begin{align*}\nonumber
\langle Z_{t},f\rangle = &f\left(x_{0}\right)+\int_{0}^{t}\int_{\mathbb{R}_+}\mathcal{G}f(x)Z_{s}\left(dx\right)ds+M_{t}^f(x_0)&\\\nonumber
 +\int_{0}^{t}\int_{E}&\mathbf{1}_{\left\{u\in V_{s^{-}}\right\}}\left(\mathbf{1}_{\left\{z\leq b\right\}}\left(f\left(\theta X_{s^-}^u \right)+ f\left((1-\theta) X_{s^-}^u \right)\right)-
\mathbf{1}_{\left\{z\leq b+d\right\}}f\left(X_{s^{-}}^{u}\right)\right) M\left(ds,du,d\theta,dz\right)\\\nonumber
 +\int_{0}^{t}\int_{E_{r1}}&\mathbf{1}_{\left\{u\in V_{s^{-}}\right\}}\left(\mathbf{1}_{\left\{z\leq \lambda(X_{s^-}^u)\right\}}\left(f\left(X_{s^-}^u+ \mathfrak{i} \right) -
f\left(X_{s^{-}}^{u}\right)\right)\right) M_{r1}\left(ds,du,d\mathfrak{i},dz\right)\\\nonumber
 +\int_{0}^{t}\int_{E_{r2}}&\mathbf{1}_{\left\{u\in V_{s^{-}}\right\}}\left(\mathbf{1}_{\left\{z\leq r\left(e^{-(b-d)s}\E_{\delta_{x_0}}[\langle Z_s,\mathrm{Id}\rangle]\right)\right\}}\left(f\left(X_{s^-}^u+ \mathfrak{p} \right) -
f\left(X_{s^{-}}^{u}\right)\right)\right) M_{r2}\left(ds,du,d\mathfrak{p},dz\right),
\end{align*}
where for all $x\geq 0$, $M_{t}^f(x)$ is a $\mathcal{F}_t$-martingale, and $Id$ denotes the identity function. 
\end{prop}

We always assume that there is one cell at time $0$.
$e^{(b-d)s}$ is thus the mean number of cells at time $s$, and
$ e^{-(b-d)s}\E[\langle Z_s,\mathrm{Id}\rangle]$
is the mean quantity of parasites in the population divided by the mean number of cells at time $s$.

The proof of Proposition \ref{pro_exi_uni} is a combination of \cite[Proposition 1]{palau2018branching} and \cite[Theorem 2.1]{marguet2016uniform}.
It is essentially the same as the proof of Proposition 1.1 in \cite{marguet2020long} and we refer the reader to this paper. The only difference is the addition of the reinfection events. Reinfection events due to the release of parasites in the cell medium after cell lyses have a bounded rate (as $r$ is bounded), thus we can construct the solution between these reinfection events, choose uniformly the cell which undergoes the infection, add the quantity of parasites in the cell, and reiterate the procedure until the next infection event of this type occurs. Due to the dynamic of the number of parasites in the cells (given by $\mathcal{G}$), the quantity of parasites may reach infinity in a cell $u_0$ for instance. But in this case, 
\begin{align*}&\mathbf{1}_{\left\{u_0\in V_{s^{-}}\right\}}\left(\mathbf{1}_{\left\{z\leq \lambda(X_{s^-}^{u_0})\right\}}\left(f\left(X_{s^-}^{u_0}+ \mathfrak{i} \right) -
f\left(X_{s^{-}}^{u_0}\right)\right)\right)\\
&=\mathbf{1}_{\left\{u_0\in V_{s^{-}}\right\}}\left(\mathbf{1}_{\left\{z\leq \lambda(\infty)\right\}}\left(f\left(\infty \right) -
f\left(\infty\right)\right)\right)=0\end{align*}
for any time $s$ after the time when $X^{u_0}$ reaches infinity. The reinfection events due to the reservoir thus do not contribute anymore to the dynamics when they concern a cell with an infinite quantity of parasites, even if $\lambda$ is not bounded.

For the sake of readability we will assume that the processes under consideration in the sequel satisfy Assumption \ref{ass_A}, but we will not indicate it.\\

We aim at investigating the long time behaviour of the infection in the cell population.
As we have explained in the introduction, the strategy to obtain information at the population level is to introduce an auxiliary process behaving as a `typical individual'. 
A slight adaptation of Theorem 3.1 in \cite{marguet2016uniform} and Proposition 1 in \cite{palau2018branching} allows us to obtain the following Many-to-One formula (see ideas of the proof on page \pageref{proof_adaptation}):

\begin{prop} \label{pro_MtO}
Let the process $Z$ be defined as in Proposition \ref{pro_exi_uni}, and recall notation \eqref{Ztdirac}. Then for $x_0 \geq 0 $ and all measurable bounded functions $F : \mathbb{D}([0, t], \bar{\R}_+ ) \to \R$, we have:
\begin{equation}\label{MtO}
 \E_{\delta_{x_0}} \left[ \sum_{u \in V_t} F(X_s^u,s \leq t) \right] = e^{(b-d)t}\E_{x_0}[F(Y_s,s\leq t)],
\end{equation}
where we recall that for $s \leq t$, $X_s^u$ denotes the ancestor of individual $u \in V_t$ at time $s$. In \eqref{MtO}, $Y=(Y_t: t \geq 0)$ is a time inhomogeneous diffusion with jumps, which is the unique non-negative strong
solution to
\begin{align} \nonumber \label{EDS_Y} 
Y_t =x_0& + \int_0^t g(Y_s)ds
+\int_0^t\sqrt{2\sigma^2 (Y_s)}dB_s +
\int_0^t\int_0^{p(Y_{s-})}\int_{\mathbb{R}_+}z\widetilde{Q}(ds,dx,dz)\\&+\nonumber
\int_0^t\int_0^{2b}\int_{0}^1(\theta-1)Y_{s-}N(ds,dx,d\theta)
+\int_0^t\int_0^{r(\E_{x_0}[Y_{s}])}\int_{\mathbb{R}_+}\mathfrak{p}N_1(ds,dx,d\mathfrak{p})\\&
+\int_0^t\int_0^{\lambda(Y_{s-})}\int_{\mathbb{R}_+}\mathfrak{i}N_2(ds,dx,d\mathfrak{i}),
\end{align}
where $N(ds,dx,d\theta)$, $N_1(ds,dx,d\mathfrak{p})$ and $N_2(ds,dx,d\mathfrak{i})$ are independent Poisson Point Measures with intensities $ds dx \kappa(d\theta)$, $ds dx \mathfrak{P}(d\mathfrak{p})$ and $ds dx \mathfrak{I}(d\mathfrak{i})$, 
and $Q$ and $B$ have been defined in \eqref{X_sans_sauts}. $N$, $N_1$, $N_2$, $Q$ and $B$ are independent.
\end{prop}

\section{Results} \label{section_results}

We will now study how the interactions between the parasites growth rate, fluctuations and positive jumps, the rate and intensity of reinfections and the cell division rate and law shape the long time behaviour of the infection.

\subsection{Controlled or uncontrolled infection}

First, we want to know whether the amount of parasites in the cells remains bounded.
To this aim, we define the function
\begin{equation} \label{def_p}
\rho(x):=\E[ \mathcal{I}]\lambda(x)+\E[ \mathcal{P}] r(x)+g(x)-bx, \quad x \geq 0,
\end{equation}
where we $\mathcal{P}$ (resp. $\mathcal{I}$) denotes a random variable with law $\mathfrak{P}$ (resp. $\mathfrak{I}$). The function $\rho$ takes into account the different mechanisms of the parasite infection: infection by the reservoir and after cell lyses, growth of parasites in cells, and cell divisions. The behaviour of $\rho(x)$ for large $x$ will tell us if the cell population may control the parasite infection on the long term.

Let us introduce the set 
\begin{equation} \label{def_mathcalA}
\mathcal{A}:= \{ a > 1, \E[\Theta^{1-a}]<\infty \},
\end{equation}
and for $a \neq 1 \in \R_+^*$ and $x>0$,
\begin{align}\label{def:Ia}
I_a(x)= ax^{-2}\int_{\mathbb{R}_+} z^2\left(\int_0^1 (1+zx^{-1}v)^{-1-a}(1-v)dv\right)\pi(dz)
\end{align}
as well as
\begin{equation} \label{def_mathcalD} \mathcal{D}(a,x):= \frac{g(x)}{x}- a \frac{\sigma^2(x)}{x^2} - p(x)I_a(x)
+ \lambda(x)\frac{\E[(1+\mathfrak{I}/x)^{1-a}]-1}{1-a} . \end{equation}

To know if the quantity of parasites can explode (reach $\infty$ in a finite time), we will contrast the two following assumptions:

\begin{enumerate}[label=\bf{(SN$\infty$)}]
\item \label{SNinfty} 
There exist $a<1$ and a non-negative function $f$ on $\mathbb{R}_+$ such that
\begin{align*}
 \mathcal{D}(a,x)=-f(x)+ o(\ln x),\quad (x\rightarrow +\infty).
\end{align*}
\end{enumerate}

\begin{enumerate}[label=\bf{(LN$\infty$)}]
\item\label{LNinfty}  There exist $a\in\mathcal{A}$, $\eta>0$ and $x_0> 0$ such that for all $x\geq x_0$
\begin{align*}
\mathcal{D}(a,x) \geq \ln x \left(\ln\ln x\right)^{1+\eta}.
\end{align*}
\end{enumerate}

Condition {\bf{(SN$\infty$)}} ensures that $Y$ does not explode, whereas under Condition {\textbf{(LN$\infty$)}}, the process $Y$ has a positive probability to reach $\infty$ in finite time (see \cite{companion} for the introduction of these conditions, being a refinement of conditions introduced in \cite{li2017general}).
Notice that under the assumptions of point $(1)$ of the following Proposition, Condition \ref{SNinfty} holds. 
Depending on the relative strengths of the functions $g, p, \sigma, \lambda$ and the parameter $b$, we may obtain contrasted long time behaviours for the infection in the cell population.

\begin{prop} \label{cor_esperance}
Let the process $Z$ be defined as in Proposition \ref{pro_exi_uni}, with $Z_0 = \delta_0$, and $N_t = \langle Z_t,1 \rangle$ denote the number of cells alive at time $t\geq 0$.
\begin{enumerate}
\item If for large $x$, $g(x)+ \E[\mathcal{I}]\lambda(x) \leq \tilde{g}x$ with $\tilde{g}<b$, then there exists a constant $C$ such that for any $K,t\geq 0$,
$$ \E_{\delta_0} \left[\mathbf{1}_{\{N_t \geq 1\}} \frac{\sum_{u \in V_t}\mathbf{1}_{\{X_t^{(u)}\geq K\}}}{N_t} \right]\leq  \frac{C}{\sqrt{K}}. $$
\item If there exists $\tilde{g}>b$ such that $g(x)+ \E[\mathcal{I}]\lambda(x)\geq \tilde{g}x$ for any $x>0$, then for every $K\geq 0$,
$$ \lim_{t \to \infty} \mathbf{1}_{\{N_t \geq 1\}} \frac{\sum_{u \in V_t}\mathbf{1}_{\{\sup_{s\leq t}X_s^{(u)}\leq K\}}}{N_t} =0, \quad \textit{in probability}. $$
\item Under Condition \ref{LNinfty} and if $\lambda(x)>0$ for any $x \in \R_+^*$,
$$ \lim_{t \to \infty} \mathbf{1}_{\{N_t \geq 1\}} \frac{\sum_{u \in V_t}\mathbf{1}_{\{X_t^{(u)}<\infty\}}}{N_t}=0, \quad a.s. $$
\end{enumerate}
\end{prop}

Notice that point {\it{(1)}} implies that for any $t\geq 0$
$$ \lim_{K \to \infty}\mathbf{1}_{\{N_t \geq 1\}} \frac{\sum_{u \in V_t}\mathbf{1}_{\{X_t^{(u)}\geq K\}}}{N_t} =0, \quad \textit{in probability}.$$

Thus, in the first case, the proportion of cells containing a large amount of parasites is small, when considering any given time. In the second case, on the contrary, when the time is large, the lineage of each cell alive at time $t$ has contained a high number of parasites at a given time with high probability. 
The third case is even more extreme as the quantity of parasites in most of the cells reaches infinity.
One could imagine that the presence of a large number of parasites in a cell disturbs the functioning of this latter. The second and third cases could therefore correspond to the death of a large fraction of the cells for example if we allowed the cell death rate to depend on the quantity of parasites.

Interestingly, the parasite growth rate $g(\cdot)$ and the infection by the reservoir $\E[\mathcal{I}]\lambda(\cdot)$ have a similar role in the quantity of parasites on the long term (increasing any of them has the same impact on the criteria of points $(1)$ and $(2)$). This highlights the importance of taking into account reinfections by the reservoir (or similarly the immigration of infected individuals) when studying an epidemics.

\subsection{The case without reservoir}

We now wonder if the cell population is able to get rid of the infection in the case where it is not reinfected by new contacts with the parasites reservoir after the first infection of the population.

\begin{prop} \label{prop_tendvers0}
Let $Z$ and $N$ be defined as in Proposition \ref{pro_exi_uni}.
Assume that $\lambda(\cdot) \equiv 0$.
\begin{enumerate}
\item Assume that $g$ is a concave function and there exist $\alpha,\eta >0$ such that
\begin{equation} \label{rho_neg} \rho(x) \leq - (\alpha x \wedge \eta), \quad x \geq 0. \end{equation}
Then for every $\eps>0$ and $Z_0 = \delta_{x_0}$ with $x_0<\infty$,
$$ \lim_{t \to \infty} \mathbf{1}_{\{N_t \geq 1\}} \frac{\sum_{u \in V_t}\mathbf{1}_{\{X_t^{(u)}> \eps\}}}{N_t}=0, \quad \textit{in probability}. $$
\item Under Condition \ref{LNinfty} and if $r(x_0)>0$,
$$ \lim_{t \to \infty} \mathbf{1}_{\{N_t \geq 1\}} \frac{\sum_{u \in V_t}\mathbf{1}_{\{X_t^{(u)}<\infty\}}}{N_t}=0, \quad \text{a.s. when } Z_0= \delta_{x_0}. $$
\end{enumerate}
\end{prop}

A wide range of long time behaviours are thus possible for the class of processes under consideration.

The concavity of $g$ seems a reasonable assumption. Indeed it means that there is a negative interaction between parasites, which may be due for instance to competition for space, or to the activation of the immune system. A logistic growth rate for example meets this assumption.

\subsection{Coming down from infinity}

We now assume that parasites from cell lyses do not reinfect other cells, for example because the immune system destroys them before they can attack other cells. However, we suppose that the cell population is frequently in contact with the parasite living in another environment (the reservoir). We investigate under what conditions on the frequency and rate of infection and on the dynamics of the parasites within the cells the cell population is able to control the size of the parasites proliferation.\\

Recall the definition the function $ \mathcal{D}(a,.)$ in \eqref{def_mathcalD}.
Then we have the following result.

\begin{prop} \label{theo_CDI}
Let the process $Z$ be defined as in Proposition \ref{pro_exi_uni}.
Assume that $r(\cdot) \equiv 0$.
\begin{enumerate}
\item If there exist $a \in \mathcal{A}$ and a non-negative function $f$ such that
\begin{equation}\label{cond_rest_infty}
\mathcal{D}(a,x) \geq f(x)+ o(\ln x), \quad (x \to \infty),
\end{equation}
then for every $K,t\geq 0$,
$$ \lim_{x \to \infty} \E_{\delta_x} \left[\mathbf{1}_{\{N_t \geq 1\}} \frac{\sum_{u \in V_t}\mathbf{1}_{\{X_t^{(u)}\leq K\}}}{N_t} \right]=0. $$
\item If there exist $0<a<1$ and $\eta>0$, such that
\begin{equation}\label{cond_comdown_infty}
\mathcal{D}(a,x) \leq -\ln x(\ln \ln x)^{1+\eta}, \quad (x \to \infty),
\end{equation}
then for every $x,t\geq 0$,
$$ \lim_{K \to \infty}\E_{\delta_x} \left[\mathbf{1}_{\{N_t \geq 1\}} \frac{\sum_{u \in V_t}\mathbf{1}_{\{X_t^{(u)}\geq K\}}}{N_t} \right]=0. $$
\end{enumerate}
\end{prop}

There is thus an interplay between the increase of the number of parasites (by growth or reinfection by the reservoir) and the fluctuations of the number of parasites (characterized by $\sigma$ and $p$), which tend to decrease the number of parasites. This interplay is fully characterized by the class of functions $(\mathcal{D}(a,.), a \neq 1)$.

\section{Proofs}

We first explain how the proofs of Proposition 1 in \cite{palau2018branching} and Theorem 3.1 in \cite{marguet2016uniform} must be modified to obtain Proposition \ref{pro_MtO}.

\begin{proof}[Proof of Proposition \ref{pro_MtO}] \label{proof_adaptation}
The existence and strong unicity of a solution to \eqref{EDS_Y}
are a consequence of Proposition 1 in \cite{palau2018branching}, and accordingly, $(Y_t, t \geq 0)$ is a $[0,\infty]$-valued process, which satisfies \eqref{EDS_Y} up to the time $\tau_n := \inf\{t \geq 0 : Y_t \geq n\}$ for all $n \geq 1$, and $Y_t = \infty$ for all
$t \geq \tau_\infty := \lim_{n\to \infty} \tau_n$. Notice that the process $Y$ is not Markovian, as $r(\E[Y_t])$ is time dependent.
In the statement of Proposition 1 in \cite{palau2018branching}, the functions do not depend on time, unlike the case of our process. However this additional
dependence does not bring any modification to the proofs (see Appendix B in \cite{marguet2020long} for more details).

The proof of the Many-to-One formula \eqref{MtO} is obtained by a slight modification of the proof of Theorem 3.1 in \cite{marguet2016uniform}. There are essentially three differences with the result of Theorem 3.1 in \cite{marguet2016uniform}. First a simplification: the division rate $b$ does not depend on the quantity of parasites in the cells, which makes equations simpler. Second, the reinfection by parasites in the medium (resp. in the reservoir), which can be seen, at time $t$, as the death of a cell $u \in V_t$ with a quantity $X_t^u$ of parasites and the birth of a cell with a quantity $X_t^u + \mathcal{I}$ of parasites, $\mathcal{I}$ following the distribution $\mathfrak{I}$ (resp. a quantity $X_t^u + \mathcal{P}$ of parasites, $\mathcal{P}$ following the distribution $\mathfrak{P}$). The point is that, contrarily to the case considered in \cite{marguet2016uniform}, this additional division rate is time inhomogeneous, and unlike the division rate $b$ it may depend on the quantity of parasites. However it does not impact the mean number of cells, and the adaptation of the proof to this case is not difficult. The last difference is that we allow for the death of cells without producing any daughter cell. Here again the adaptation is fairly simple, because this death rate, $d$, does not depend on the quantity of parasites in the cell. To fix ideas, these modifications would lead to the following equivalent of Equation (3.8) in \cite{marguet2016uniform}:
\begin{align*}
\mu_{t_0,s}(x_0,f)=& f(t_0,x_0) + \int_{t_0}^s \int_{\R_+} \left[ \mathcal{G}f(r,x)-(b-d)f(r,x)+ \partial_rf(r,x) \right]\mu_{t_0,r}(x_0,dx)dr\\
&+ \int_{t_0}^s \int_{\R_+} \left[ 2b \int_0^1 f(r,\theta x)\kappa(d\theta)-(b+d)f(r,x) \right] \mu_{t_0,r}(x_0,dx)dr\\
&+ \int_{t_0}^s \int_{\R_+} \left[r(\E[Y_r|Y_{t_0}=x_0])\int_{\R_+}(f(r,x+\mathfrak{p})-f(r,x))\mathfrak{P}(d\mathfrak{p}) \right]\mu_{t_0,r}(x_0,dx)dr\\
&+ \int_{t_0}^s \int_{\R_+} \left[\lambda(x)\int_{\R_+}(f(r,x+\mathfrak{i})-f(r,x))\mathfrak{I}(d\mathfrak{i}) \right]\mu_{t_0,r}(x_0,dx)dr.
\end{align*}
We see that the terms involving $d$ cancel out. This explains why the death rate does not impact the dynamics of $Y$. Intuitively this is due to the fact that the death rate is constant. Hence all individuals die with the same probability and the behaviour of a 'typical individual' is not modified.
\end{proof}

We may now prove the results on the long time behaviour of the infection presented in Section \ref{section_results}. \\

Let us recall the definition of $(Y_t, t \geq 0)$ in \eqref{EDS_Y}, and introduce the process $(\bar{Y}_t, t \geq 0)$, solution to the SDE
\begin{align} \nonumber \label{EDS_barY}
\bar{Y}_t =&  \int_0^t g(\bar{Y}_s)ds
+\int_0^t\sqrt{2\sigma^2 (\bar{Y}_s)}dB_s +
\int_0^t\int_0^{p(\bar{Y}_{s-})}\int_{\mathbb{R}_+}z\widetilde{Q}(ds,dx,dz)\\& \nonumber+\int_0^t\int_0^{2b}\int_0^1(\theta-1)\bar{Y}_{s-}N(ds,dx,d\theta)\\&
+\int_0^t\int_0^{r(\bar{y})}\int_{\mathbb{R}_+}\mathfrak{p}N_1(ds,dx,d\mathfrak{p})
+\int_0^t\int_0^{\lambda( \bar{Y}_{s-})}\int_{\mathbb{R}_+}\mathfrak{i}N_2(ds,dx,d\mathfrak{i}),
\end{align}
where the Brownian $B$ and the Poisson random measures $Q$, $N$, $N_1$ and $N_2$ are the same as in the SDE \eqref{EDS_Y}, and $\bar{y}$ is a non-negative real number, which will be chosen later.

Moreover, for the sake of readability we introduce the following notations:
\begin{itemize}
\item For $x>0$, 
\begin{equation}\label{def_tau}\tau^+(x):= \inf\{ t \geq 0, Y_t\geq x \} \quad \text{and} \quad \tau^-(x):= \inf\{ t \geq 0, Y_t\leq x \}.\end{equation}
\item 
For $x>0$, 
\begin{equation*}\bar{\tau}^+(x):= \inf\{ t \geq 0, \bar{Y}_t\geq x \} \quad \text{and} \quad \bar{\tau}^-(x):= \inf\{ t \geq 0, \bar{Y}_t\leq x \}.\end{equation*}
\item \begin{equation}\label{def_tau_infty} \tau^+(\infty) = \lim_{x \to \infty} \tau^+(x) \quad \text{and} \quad 
 \bar{\tau}^+(\infty) = \lim_{x \to \infty} \bar{\tau}^+(x). \end{equation}
\end{itemize}

Before giving the main proofs, we will derive a lemma, which will be useful on several occasions. It concerns the probability for the processes $Y$ and $\bar{Y}$ to reach infinity in a finite time.

\begin{lemma}\label{lem_reach_infty}
\begin{itemize}
\item[i)] If Condition \ref{SNinfty} holds, then 
$$\mathbb{P}_x\left(\tau^+(\infty)<\infty\right)=\mathbb{P}_x\left(\bar{\tau}^+(\infty)<\infty\right)=0 \quad \text{for all} \quad x>0.$$ 
\item[ii)] If Condition \ref{LNinfty} holds then $\mathbb{P}_x\left(\tau^+(\infty)<\infty\right) >0$ and $\mathbb{P}_x\left(\bar{\tau}^+(\infty)<\infty\right)>0$ for all large enough $x>0$.
\end{itemize}
\end{lemma}

\begin{proof}
Let us recall the definition of $\mathcal{A}$ in \eqref{def_mathcalA}
and introduce the set of functions $G_a^{(t)}$ given for $a \in \mathcal{A} \cup (0,1)$, $t \geq 0$ and $x_0 \geq 0,x>0$ by
\begin{align}\label{eq:Ga1}
G_a^{(t)}(x_0,x) : = & (a-1)\Bigg(\frac{g(x)}{x}-a\frac{\sigma^2(x)}{x^2}-2b\frac{1-\E[\Theta^{1-a}]}{1-a}-p(x)I_a(x)\nonumber
\\&+ \lambda(x)\E\left[\frac{\left(1+\mathcal{I}/x\right)^{1-a}-1}{1-a}\right]\Bigg)- r(\E_{x_0}[Y_t])\E\left[\left(1+\frac{\mathcal{P}}{x}\right)^{1-a}-1\right],
\end{align}
where $\delta_{x_0}$ is the initial condition of $Z$ and $I_a$ has been defined in 
\eqref{def:Ia}.
For all $b>c>0$, let $T = \tau^-(c)\wedge \tau^+(b)$. Then, for all $a\in\mathcal{A}\cup (0,1)$, the process 
\begin{equation} \label{def_Za}
Z^{(a)}_{t\wedge T}:=\left(Y_{t\wedge T}\right)^{1-a}\exp\left(\int_0^{t\wedge T} G_a^{(s)}\left(x_0,Y_s\right)ds\right)
\end{equation}
is a $\mathcal{F}_t$-martingale. The proof of this property is the same as the proof of \cite[Lemma 5.1]{li2017general} or \cite[Lemma 7.1]{companion}.

Now we see that Condition \ref{LNinfty} is equivalent to 
the existence of $a\in\mathcal{A}$ and $\eta>0$ such that for all $t \geq 0$
$$ G_a^{(t)}(x_0,x) \geq \ln x \left(\ln(\ln x)\right)^{1+\eta},\quad (x\rightarrow +\infty), $$
and that Condition \ref{SNinfty} is equivalent to 
the existence of $a<1$ and a non-negative function $f$ on $\mathbb{R}_+$ such that for all $t\geq 0$,
$$ G_a^{(t)}(x_0,x) \geq f(x)+ o(\ln x),\quad (x\rightarrow +\infty).$$

Indeed, the second to last term in \eqref{eq:Ga1} has no influence as $r$ is a bounded function: this term is thus bounded on $\R_+$.

Introducing similarly the set of functions $\bar{G}_a$, for $a \in \mathcal{A} \cup (0,1)$, $x_0 \leq \bar{y}$ and $x>0$ via
\begin{align*}
\bar{G}_a(x) : = & (a-1)\left(\frac{g(x)}{x}-a\frac{\sigma^2(x)}{x^2}-2b\frac{1-\E[\Theta^{1-a}]}{1-a}-p(x)I_a(x)+ \lambda (x)\E\left[\frac{\left(1+\mathcal{I}/x\right)^{1-a}-1}{1-a}\right]\right)\nonumber\\
&- r(\bar{y})\E\left[\left(1+\frac{\mathcal{P}}{x}\right)^{1-a}-1\right],
\end{align*}
we obtain that for all $a\in\mathcal{A}\cup (0,1)$, the process 
$$
\bar{Z}^{(a)}_{t\wedge T}:=\left(\bar{Y}_{t\wedge T}\right)^{1-a}\exp\left(\int_0^{t\wedge T} \bar{G}_a\left(Y_s\right)ds\right)
$$ 
is a $\mathcal{F}_t$-martingale, and we have the same equivalences with $\bar{G}_a(\cdot)$ in place of $G_a^{(t)}(x_0,\cdot)$.

The proof of Lemma \ref{lem_reach_infty} is therefore a direct application of points \textit{i)} and \textit{ii)} of \cite[Theorem 4.1]{companion}.
\end{proof}

The proof strategies are similar for the three main results of this paper. We first derive a property on the auxiliary process  and then show how we can infer the long time behaviour of the infection at the cell population scale.

\begin{proof}[Proof of Proposition \ref{cor_esperance}]

Let us begin with the proof of point $(1)$.

Recall the definition of $\rho$ in \eqref{def_p}.
We first make the simplifying assumption that 
there exists $x_0$ such that 
$$ g(x)+ \E[\mathcal{I}]\lambda (x) = \tilde{g}x , \quad x \geq x_0.$$
In particular, for $x \geq x_0$
$$ \rho(x)= (\tilde{g}-b)x + \E[\mathcal{P}]r(x) , $$
which goes to $-\infty$ when $x$ goes to infinity under the assumptions of point (1) ($\tilde{g}<b$).
This implies the existence of $\bar{y}\geq x_0$ and $\eps>0$ such that for any $x \geq \bar{y}$,
\begin{equation}\label{def_y}
\rho(x) \leq -M-(b-\tilde{g})\eps,
\end{equation}
where
$$ M:= \sup_{x \leq x_0}\left\{g(x)+ \E[\mathcal{I}]\lambda(x)\right\}. $$
We choose such a real number $\bar{y}$ in the definition of $\bar{Y}$ in \eqref{EDS_barY}.

We will first derive the following property for the auxiliary process $Y$:
\begin{equation*} \E[Y_t]\leq \bar{y}, \quad \text{for all}\quad t \geq 0 \end{equation*}
by proving successively that
\begin{enumerate}
\item for every $t\geq 0$, $\E[\bar{Y}_t]\leq \bar{y}$,
\item for every $t\geq 0$, $Y_t \leq \bar{Y}_t$ a.s.
\end{enumerate}

Recall \eqref{EDS_barY}. Applying Itô formula with jumps (see for instance \cite{ikeda1989} Th 5.1 p.66), we get
\begin{align*} 
e^{(b-\tilde{g})t}\bar{Y}_t =& 
\int_0^te^{(b-\tilde{g})s}\sqrt{2\sigma^2 (\bar{Y}_s)}dB_s +
\int_0^t\int_0^{p(\bar{Y}_{s-})}\int_{\mathbb{R}_+}e^{(b-\tilde{g})s}z\widetilde{Q}(ds,dx,dz)\\& \nonumber
+\int_0^t\int_0^{2b}\int_0^1e^{(b-\tilde{g})s}(\theta-1)\bar{Y}_{s-}\tilde{N}(ds,dx,d\theta) \\& \nonumber+ \int_0^t e^{(b-\tilde{g})s}(g(\bar{Y}_s)+ \E[\mathcal{I}]\lambda(\bar{Y}_s)-\tilde{g}\bar{Y}_s)\mathbf{1}_{\{\bar{Y}_s \leq x_0\}}ds \\&
+\int_0^t\int_0^{r(\bar{y})}\int_{\mathbb{R}_+}e^{(b-\tilde{g})s}\mathfrak{p}N_1(ds,dx,d\mathfrak{p})
+\int_0^t\int_0^{\lambda(\bar{Y}_{s-})}\int_{\mathbb{R}_+}e^{(b-\tilde{g})s}\mathfrak{i}\tilde{N}_2(ds,dx,d\mathfrak{i})\\
=&M^{(loc)}_t+\int_0^t\int_0^{r(\bar{y})}\int_{\mathbb{R}_+}e^{(b-\tilde{g})s}\mathfrak{p}N_1(ds,dx,d\mathfrak{p})
\\& \nonumber+ \int_0^t e^{(b-\tilde{g})s}(g(\bar{Y}_s)+ \E[\mathcal{I}]\lambda(\bar{Y}_s)-\tilde{g}\bar{Y}_s)\mathbf{1}_{\{\bar{Y}_s \leq x_0\}}ds ,
\end{align*}
where $(M^{(loc)}_{t\wedge \tau^+(x)}, t \geq 0)$ is a martingale for any $x>0$.
Indeed this is a local martingale and we can check using \cite[Theorem 51 p.38]{protter2005stochastic} as in the proof of Lemma 5.1 in \cite{li2017general} that it is a martingale.
Taking the expectation yields
\begin{equation} \label{comput_esp_barY} \E [e^{(b-\tilde{g})(t\wedge \tau^+(x))}\bar{Y}_{t\wedge \bar{\tau}^+(x)} ]\leq \left(r(\bar{y})\E[\mathcal{P}]+ M\right)\E \left[ \frac{e^{(b-\tilde{g})(t\wedge \bar{\tau}^+(x))}}{b-\tilde{g}} \right]. \end{equation}
As $g(x)+ \E[\mathcal{I}]\lambda(x)=\tilde{g}x$ for large $x$, condition \textbf{(SN$\infty$)} is satisfied and thus according to Lemma \ref{lem_reach_infty}, 
$$\lim_{x \to \infty} \bar{\tau}^+(x)=\infty. $$
As a consequence, letting $x$ go to infinity, we obtain by Fatou's Lemma:
\begin{align*} e^{(b-\tilde{g})t}\E [\bar{Y}_{t} ]=& \E [\liminf_{x\to \infty}e^{(b-\tilde{g})(t\wedge \tau^+(x))}\bar{Y}_{t\wedge \bar{\tau}^+(x)} ]\\
\leq & \liminf_{x\to \infty}\E [e^{(b-\tilde{g})(t\wedge \tau^+(x))}\bar{Y}_{t\wedge \bar{\tau}^+(x)} ]\\
\leq& 
 \left(r(\bar{y})\E[\mathcal{P}]+M\right)\left( \frac{e^{(b-\tilde{g})t}-1}{b-\tilde{g}} \right), \end{align*}
and thus
\begin{equation*} \E [\bar{Y}_{t} ]\leq \frac{r(\bar{y})\E[\mathcal{P}]+ M}{b-\tilde{g}}\leq \bar{y}-\eps. \end{equation*}
 This concludes the proof of step (1).\\

Let us now move on to step (2). By unicity of the solutions to \eqref{EDS_Y} and \eqref{EDS_barY} and as the rate of positive jumps $p$ is non decreasing with the process value, we have that for $t\geq 0$:
$$ Y_t \leq \bar{Y}_t \quad \text{a.s. if} \quad \E[Y_s]\leq \bar{y} \quad \forall s \leq t.  $$
We will show that this last inequality is satisfied for every $s\geq 0$.\\

Notice first that with computations similar to the ones to get \eqref{comput_esp_barY}, we obtain:
\begin{align} \label{comput_esp_Y} \E [e^{(b -\tilde{g})(t\wedge \tau^+(x))}Y_{t\wedge \tau^+(x)} ]= & 
\E[\mathcal{P}] \int_0^t \E \left[\mathbf{1}_{s\leq \tau^+(x)}e^{(b -\tilde{g})s}r(\E[Y_{s}]) \right]ds \\
&+ \int_0^t \E \left[\mathbf{1}_{s\leq \tau^+(x)}e^{(b -\tilde{g})s}(g(Y_{s})+ \E[\mathcal{I}]\lambda(Y_s)-\tilde{g}Y_{s}) \mathbf{1}_{Y_s \leq x_0} \right]ds \nonumber. \end{align}
Hence for a given $x$, $\E [e^{(b-\tilde{g})(t\wedge \tau^+(x))}Y_{t\wedge \tau^+(x)} ]$ evolves continuously with $t$. This justifies the existence of $\mu(x)$ and $\upsilon(x)$ defined as follows: assume that there exist $(x,t) \in \R_+^2$ such that
$$ \E \left[e^{(b-\tilde{g})(t\wedge \tau^+(x))}\bar{Y}_{t\wedge \bar{\tau}^+(x)} \right]e^{-(b-\tilde{g})t}= \bar{y}+\eps, $$
and choose the smaller $t$ realising this equality:
$$ \mu(x):= \inf \left\{ t \geq 0, \E \left[e^{(b-\tilde{g})(t\wedge \tau^+(x))}\bar{Y}_{t\wedge \bar{\tau}^+(x)} \right]e^{-(b-\tilde{g})t}= \bar{y}+\eps \right\} $$
(if $\mu(x)= \infty$, we may jump directly to Equation \eqref{maj_ybareps}).
Now, define $\upsilon(x)\geq \mu(x)$ as follows:
$$ \upsilon(x):= \inf \left\{ t \geq \mu(x), \E \left[e^{(b-\tilde{g})(t\wedge \tau^+(x))}\bar{Y}_{t\wedge \bar{\tau}^+(x)} \right]e^{-(b-\tilde{g})t} \notin [\bar{y},\bar{y}+2\eps] \right\}, $$
where $\eps >0$.

Notice that \eqref{comput_esp_Y} may be rewritten, for $\mu(x) \leq t \leq \upsilon(x)$,
\begin{multline*}  \E [e^{(b-\tilde{g})(t\wedge \tau^+(x)}Y_{t\wedge \tau^+(x)} ]-\E [e^{(b-\tilde{g})(\mu(x)\wedge \tau^+(x))}Y_{\sigma\wedge \tau^+(x)} ] = \\ \int_{\mu(x)}^t\E \Big[ \mathbf{1}_{s\leq \tau^+(x)} e^{(b-\tilde{g})s} \Big(
\E[\mathcal{P}] r(\E[e^{(b-\tilde{g})(s\wedge \tau^+(x))}Y_{s\wedge \tau^+(x)}]e^{-(b-\tilde{g})s}) \\+(g(Y_{s})+\E[\mathcal{I}]\lambda(Y_s)-\tilde{g}Y_{s}) \mathbf{1}_{Y_s \leq x_0}\Big) \Big]ds\\
=\int_{\mu(x)}^t\E \Big[ \mathbf{1}_{s\leq \tau^+(x)} e^{(b-\tilde{g})(t\wedge \tau^+(x))s} \Big(
 \rho(\E[e^{(b-\tilde{g})(s\wedge \tau^+(x))}Y_{s\wedge \tau^+(x)}]e^{-(b-\tilde{g})s}) \\+(b-\tilde{g})\E[e^{(b-\tilde{g})(s\wedge \tau^+(x))}Y_{s\wedge \tau^+(x)}]e^{-(b-\tilde{g})s} +(g(Y_{s})+\E[\mathcal{I}]\lambda(Y_s)-\tilde{g}Y_{s}) \mathbf{1}_{Y_s \leq x_0}\Big) \Big]ds.\end{multline*}
Using \eqref{def_y} and the fact that $\bar{y}\geq x_0$, we obtain:
\begin{multline*}  \E [e^{(b-\tilde{g})(t\wedge \tau^+(x))}Y_{t\wedge \tau^+(x)} ] \leq \E [e^{(b-\tilde{g})(\mu(x)\wedge \tau^+(x))}Y_{\sigma\wedge \tau^+(x)} ] \\
+
(b-\tilde{g})(t\wedge \tau^+(x))\int_{\mu(x)}^t \E \left[\mathbf{1}_{s\leq \tau^+(x)}\E[e^{(b-\tilde{g})(s\wedge \tau^+(x))}Y_{s\wedge \tau^+(x)}] \right]ds  .
 \end{multline*}
Applying Grömwall's Lemma then yields
$$\E [e^{(b-\tilde{g})(t\wedge \tau^+(x))}Y_{t\wedge \tau^+(x)} ] \leq  
\E [e^{(b-\tilde{g})(\mu(x)\wedge \tau^+(x))}Y_{\mu(x)\wedge \tau^+(x)} ]e^{(b-\tilde{g})(t-\mu(x))}, $$
or equivalently
$$\E [e^{(b-\tilde{g})(t\wedge \tau^+(x))}Y_{t\wedge \tau^+(x)} ]e^{-(b-\tilde{g})t} \leq  
\E [e^{(b-\tilde{g})(\mu(x)\wedge \tau^+(x))}Y_{\mu(x)\wedge \tau^+(x)} ]e^{-(b-\tilde{g})\mu(x)}=\bar{y}+\eps.$$
This implies that for any $(t,x) \in \R_+^2$, 
\begin{equation} \label{maj_ybareps}\E [e^{(b-\tilde{g})(t\wedge \tau^+(x))}Y_{t\wedge \tau^+(x)} ]e^{-(b-\tilde{g})t} \leq \bar{y}+\eps,\end{equation}
and thus for any $t\geq 0$, by Fatou's Lemma, as Condition \ref{SNinfty} is satisfied
\begin{multline*}\E[Y_t]= \E[e^{(b-\tilde{g})t}Y_t]e^{-(b-g)t}=\E[\liminf_{x\to \infty}e^{(b-\tilde{g})(t\wedge \tau^+(x))}Y_{t\wedge \bar{\tau}^+(x)}]e^{-(b-g)t} \\
\leq \liminf_{x\to \infty}\E[e^{(b-\tilde{g})(t\wedge \tau^+(x))}Y_{t\wedge \bar{\tau}^+(x)}]e^{-(b-\tilde{g})t} \leq \bar{y}+\eps. \end{multline*}
This concludes step $(2)$ in the case $g(x)+\E[\mathcal{I}]\lambda(x) =\tilde{g}x$ for $x \geq x_0$.
Now, as $p$ and $r$ are non decreasing, if we choose the same Poisson point processes $N$, $Q$, $N_1$ and $N_2$, and Brownian motion $B$ we may couple two solutions $Y^{(1)}$ and $Y^{(2)}$ of the SDE \eqref{EDS_Y} with respective growth rate functions for the parasites $g_1$ and $g_2$ satisfying 
$$ g_1(x) \leq g_2(x) \quad \text{for all} \ x \geq 0 $$
and respective reservoir infection rates 
$\lambda_1$ and $\lambda_2$ satisfying 
$$ \lambda_1(x) \leq \lambda_2(x) \quad \text{for all} \ x \geq 0 $$
such that
$$ Y^{(1)}_t \leq Y^{(2)}_t \ a.s. \ \forall t \geq 0 $$
as soon as $Y^{(1)}_0=Y^{(2)}_0$.
The choices $g_1(x)=g(x)$, $\lambda_1(x)=\lambda(x)$ and $g_2$ and $\lambda_2$ continuous satisfying
$$\left\{ \begin{array}{llll}
g_2(x)+ \E[\mathcal{I}]\lambda_2(x)&\geq&g(x)+ \E[\mathcal{I}]\lambda(x) & \text{for $x \leq x_0$}\\
&=&\tilde{g}x & \text{for $x \geq x_0$},
\end{array}\right.$$
allow us to extend the proof of step $(2)$ to the assumptions of point $(1)$ of Proposition \ref{cor_esperance}.
\\

We may now prove point $(1)$ of Proposition \ref{cor_esperance}.
Recall that $\E[Y_t]\leq \bar{y}$ for every $t \geq 0$. In particular, from Markov's inequality we have for every $K>0$,
\begin{equation*}
 \P(Y_t \geq K) \leq \frac{\E[Y_t]}{K}\leq \frac{\bar{y}}{K}. \end{equation*}
 And from the Many-to-One formula \eqref{MtO} applied to the bounded function $x \mapsto \mathbf{1}_{\{x \geq K\}}$, we deduce
 \begin{equation} \label{MtO_for_espfinie} e^{-(b-d)t}\E \left[ \mathbf{1}_{\{N_t \geq 1\}}\sum_{u \in V_t} \mathbf{1}_{\{X_t^u \geq K\}} \right] \leq \frac{\bar{y}}{K}. \end{equation}
According to \cite{durrett2015branching}, if we denote by $\alpha_t$ and $\beta_t$ the following variables,
$$ \alpha_t:= \frac{de^{(b-d)t}-d}{be^{(b-d)t}-d}\quad \text{and} \quad 
\beta_t:= \frac{be^{(b-d)t}-b}{be^{(b-d)t}-d}, $$
we have for $t \geq 0$,
$$ \P(N_t=0)= \alpha_t \quad \text{and for } n \geq 1, \quad \P(N_t=n)= (1-\alpha_t)(1-\beta_t)\beta_t^{n-1}.  $$
We thus obtain, for $t \geq 0$ and $K \in \N$,
\begin{align*}
\P \left( N_t e^{-(b-d)t}\leq \frac{1}{\sqrt{K}}, N_t \geq 1 \right) =
& \sum_{i=1}^{\lfloor e^{(b-d)t}/\sqrt{K} \rfloor} \P(N_t=i)\\
=& (1-\alpha_t)\left(1- \beta_t^{\lfloor e^{(b-d)t}/\sqrt{K} \rfloor} \right) \\
=& \left( 1-\frac{d}{b}+ O(e^{-(b-d)t}) \right) \left( \frac{b-d}{b \sqrt{K}}+ O\left(   \frac{e^{-(b-d)t}}{\sqrt{K}} \right) \right)\\
&= \frac{1}{\sqrt{K}} \left(\frac{b-d}{b}\right)^2+ O\left(   \frac{e^{-(b-d)t}}{\sqrt{K}}\right).
\end{align*}
We thus get, using \eqref{MtO_for_espfinie}, 
\begin{align*}
\E \left[ \mathbf{1}_{\{N_t \geq 1\}}\frac{\sum_{u \in V_t} \mathbf{1}_{\{X_t^u \geq K\}}}{N_t} \right] = & \E \left[\frac{\sum_{u \in V_t} \mathbf{1}_{\{X_t^u \geq K\}}}{e^{(b-d)t}}\frac{ \mathbf{1}_{\{N_t \geq 1\}}}{N_te^{-(b-d)t}} \right] \\
\leq & \P \left( N_t e^{-(b-d)t}\leq \frac{1}{\sqrt{K}}, N_t \geq 1 \right)+\frac{\bar{y}}{K} \sqrt{K}\\
=& \frac{1}{\sqrt{K}} \left(\left(\frac{b-d}{b}\right)^2+\bar{y}\right)+ O\left(   \frac{e^{-(b-d)t}}{\sqrt{K}}\right), 
\end{align*}
which ends the proof of point $(1)$.\\

Let us now prove point $(2)$. 
Again we make simplifying assumptions and will obtain the general case by coupling. We thus assume that $g(x)+\E[\mathcal{I}]\lambda(x)=\tilde{g} x$ with $\tilde{g}>b$.
In this case, from similar computations to the ones to get \eqref{comput_esp_barY}, we derive the following series of inequalities for $0 \leq y \leq x$,
\begin{align} \E_y [e^{(b-\tilde{g})(t\wedge \tau^+(x))}Y_{t\wedge \tau^+(x)} ]=y+  
\E[\mathcal{P}] \int_0^t \E_y \left[\mathbf{1}_{s\leq \tau^+(x)}e^{(b-\tilde{g})s}r(\E_y[Y_{s}]) \right]ds\geq y.  \label{min_liminf} \end{align}

As $g(x)+\E[\mathcal{I}]\lambda(x)=\tilde{g}x$, condition \textbf{(SN$\infty$)} is satisfied and thus according to Lemma \ref{lem_reach_infty}, 
$$ \limsup_{x \to \infty} \tau^+(x)=\infty, \quad \P_y-a.s. $$

We will now prove that \eqref{min_liminf} implies that for all $x \geq 0$,
\begin{equation} \label{not_equal_1}  \P_y (\tau^+(x)=\infty)<1. \end{equation}

To this aim we will make a reductio ad absurdum by assuming that there exists $x_0 \geq 0$ such that
\begin{equation*}  \P_y (\tau^+(x_0)=\infty)=1. \end{equation*}

Notice that $x_0$ may be chosen larger than $y$ without loss of generality. We thus have, for any $t \geq 0$,
$$ Y_{t\wedge \tau^+(x_0)}\leq x_0 \quad \text{and} \quad 
e^{(b-\tilde{g})(t\wedge \tau^+(x_0))}=e^{(b-\tilde{g})}t, \quad a.s.  $$
Hence, taking $x=x_0$ in \eqref{min_liminf} we obtain that for any $t\geq 0$,
$$ x_0 e^{(b-\tilde{g})}t \geq y. $$
But by assumption, the left hand term goes to $0$ when $t$ goes to infinity, which leads to a contradiction.
We thus conclude that \eqref{not_equal_1} holds.

As a consequence, for every $x\geq 0$, there exist $\alpha_x>0$ and $t_x<\infty$ such that
$$ \P(\tau^+(x)\leq t_x)\geq  \alpha_x. $$

Notice that as the functions $p$ and $r$ are non-decreasing, the solution of \eqref{EDS_Y} is non decreasing with the initial condition. In particular, for any $y\geq 0$, 
$$ \P(\tau^+(x)> t_x)= \P_0(\tau^+(x)> t_x) \geq \P_y(\tau^+(x)> t_x). $$
Thus for any $n \in \N^*$,
\begin{align*}
\P(\tau^+(x)> nt_x)&= \P\left(\tau^+(x)> (n-1)t_x\right)\P\left(\tau^+(x)> nt_x|\tau^+(x)> (n-1)t_x\right)\\
& \leq \P\left(\tau^+(x)> (n-1)t_x\right)\sup_{y \geq 0}\P_y\left(\tau^+(x)> t_x\right)\\
& = \P\left(\tau^+(x)> (n-1)t_x\right)\P_0\left(\tau^+(x)> t_x\right)\\
& \leq \P\left(\tau^+(x)> (n-1)t_x\right)(1-\alpha_x).
\end{align*}
By recurrence, 
$$ \P(\tau^+(x)> nt_x)\leq \left( 1-\alpha_x\right)^n, $$
and taking the limit when $n$ goes to infinity, we obtain that
$$ \P(\tau^+(x)<\infty)=1-\P(\tau^+(x)=\infty)=1 . $$
In particular, for any $K,\eps>0$, there exists $t_K<\infty$ such that for every $t \geq t_K$
\begin{equation}\label{quantif_sup}
 \P \left( \sup_{s \leq t} Y_s \leq K \right)\leq \P \left( \sup_{s \leq t_K} Y_s \leq K \right) =  \P \left( \tau^+(K)\geq t_K \right)\leq \eps.
\end{equation}
Now from the Many-to-One formula \eqref{MtO}, we know that for $t,K \geq 0$,
$$ \E \left[\mathbf{1}_{\{N_t \geq 1\}} \frac{\sum_{u \in V_t}\mathbf{1}_{\{\sup_{s\leq t}X_s^{(u)}\leq K\}}}{e^{(b-d)t}} \right] = \P \left(\sup_{s\leq t}Y_s\leq K\right). $$
Hence from \eqref{quantif_sup},
$$ \frac{\sum_{u \in V_t} \mathbf{1}_{\{\sup_{s\leq t}X_s^{(u)}\leq K\}} }{e^{(b-d) t}}\xrightarrow[t\rightarrow \infty]{} 0 \quad \text{in probability}. $$

Moreover, the fact that $(N_t,t\geq 0)$ is a birth and death process with individual death rate $d$ and individual birth 
rate $b$ also entails that $N_te^{-(b-d) t}$ converges in probability to an exponential random 
variable with parameter $b-d$ on the event of survival, when $t$ goes to infinity (see \cite{durrett2015branching} for instance). Hence, we have
\begin{align*}
\mathbf{1}_{\{N_t \geq 1\}}\frac{\sum_{u \in V_t} \mathbf{1}_{\{\sup_{s\leq t}X_s^{(u)}\leq K\}} }{N_t} = \frac{\sum_{u \in V_t} \mathbf{1}_{\{\sup_{s\leq t}X_s^{(u)}\leq K\}} }{e^{(b-d) t}}\times \frac{\mathbf{1}_{\{N_t \geq 1\}}}{N_te^{-(b-d) t}}\xrightarrow[t\rightarrow \infty]{} 0 \quad \text{in probability}.
\end{align*}
To end the proof of point $(2)$ we use again the fact that we may couple a solution $Y^{(1)}$ of \eqref{EDS_Y} with $g_1(x)+ \E[\mathcal{I}]\lambda_1(x) =\tilde{g}x$ with a solution $Y^{(2)}$ with $g_2(x)+ \E[\mathcal{I}]\lambda_2(x) \geq\tilde{g}x$ with the same initial condition in such a way that 
$$ Y^{(1)}_t \leq Y^{(2)}_t \quad \forall t\geq 0 \quad a.s. $$
\vspace{.2cm}

Let us end with the proof of point $(3)$. To this aim, we introduce the process $\tilde{Y}$ defined as the unique strong solution to the following SDE
\begin{align}  \label{EDS_tildeY}
\tilde{Y}_t =&  \int_0^t g(\tilde{Y}_s)ds
+\int_0^t\sqrt{2\sigma^2 (\tilde{Y}_s)}dB_s +
\int_0^t\int_0^{p(\tilde{Y}_{s-})}\int_{\mathbb{R}_+}z\widetilde{Q}(ds,dx,dz)\\& \nonumber+\int_0^t\int_0^{2b}\int_{\mathbb{R}_+}(\theta-1)\tilde{Y}_{s-}N(ds,dx,d\theta)
,
\end{align}
where the Brownian $B$ and the Poisson random measures $N$ and $Q$ are the same as in the SDE \eqref{EDS_Y} solved by the process $(Y_t, t \geq 0)$.
Following the same proof as the one of Lemma \ref{lem_reach_infty} (it is enough to take $r\equiv 0$), we may check that Lemma \ref{lem_reach_infty} also holds for the process $\tilde{Y}$.
Moreover, a close look at the proof of \cite[Theorem 2.8]{li2017general} shows that the probability for $\tilde{Y}$ to reach infinity in a given time is uniformly lower bounded for an initial condition large enough in the following sense:
\begin{equation} \label{lowerbound_explo}
\exists \ \alpha >0, \ x_1,T_1 <\infty, \quad \inf_{x \geq x_1} \P_x(\tilde{\tau}^+(\infty)<T_1)\geq \alpha,
\end{equation}
where $\tilde{\tau}^+(\infty)$ is defined as $\tau^+(\infty)$ in \eqref{def_tau_infty} with the process $\tilde{Y}$ instead of $Y$.
To be more precise, it is due to the fact that if we introduce the function 
$$t(x):= \left(\ln (\ln(x^{(1-\delta)}))\right)^{-1-\eta}, $$
where $\delta \leq (3-2a)^{-1}$, and the stopping times
$$ \tilde{\tau}_0=0 \quad \text{and} \quad \tilde{\tau}_{n+1}:= \tilde{\tau}^+ \left( \left( \tilde{Y}_{\tilde{\tau}_n} \vee 1 \right)^{1+\delta} \right)+\tilde{\tau}_n \quad \text{for} \quad n \in \N, $$
then under $\P_{\eps^{-1}}$, if $\tilde{\tau}_n<\infty$ for all $n\geq 1$,
$$ \sum_{n=1}^\infty t(\tilde{Y}_{\tilde{\tau}_n})\leq \left( n \ln (1+\delta) \right)^{-(1+\delta)}=:\mathfrak{B}<\infty, $$
and for small enough $\eps$, 
$$ \P_{\eps^{-1}} \left(\tilde{\tau}^+(\infty)<\mathfrak{B}\right)\geq \prod_{n=1}^\infty \left(1- \eps^{(a-1)\delta(1+\delta)^n}\right)>0. $$
As the last quantity is decreasing with $\eps$, we get \eqref{lowerbound_explo}.\\

Let us now introduce the process $\tilde{\tilde{Y}}$ defined as the unique strong solution to the following SDE
\begin{align} \nonumber \label{EDS_tildetildeY}
\tilde{\tilde{Y}}_t =& \int_0^t g(\tilde{\tilde{Y}}_s)ds
+\int_0^t\sqrt{2\sigma^2 (\tilde{\tilde{Y}}_s)}dB_s +
\int_0^t\int_0^{p(\tilde{\tilde{Y}}_{s-})}\int_{\mathbb{R}_+}z\widetilde{Q}(ds,dx,dz)\\& +\int_0^t\int_0^{2b}\int_0^1(\theta-1)\tilde{\tilde{Y}}_{s-}N(ds,dx,d\theta) + \int_0^t\int_0^{\underline{\lambda}}\int_{\mathbb{R}_+}\mathfrak{i}N_2(ds,dx,d\mathfrak{i})
,
\end{align}
where $\underline{\lambda}:= \min_{x \leq 3x_1}\lambda(x)$,  the Brownian $B$ and the Poisson random measures $N$, $Q$ and $N_2$ are the same as in the SDE \eqref{EDS_Y} solved by the process $(Y_t, t \geq 0)$.
Now denote: 
$$A(t):=\int_0^t\sqrt{2\sigma^2 (\tilde{\tilde{Y}}_s)}dB_s+\int_0^t\int_0^{p(\tilde{\tilde{Y}}_{s-})}\int_{\mathbb{R}_+}z\widetilde{Q}(ds,dx,dz)
+\int_0^t\int_0^{2b}\int_{0}^1(\theta-1)\tilde{\tilde{Y}}_{s-}N(ds,dx,d\theta)$$
and
$$ B(t):=\int_0^t\int_0^{\underline{\lambda}}\int_{\mathbb{R}_+}\mathfrak{i}N_2(ds,dx,d\mathfrak{i}).$$

If we take $\mu>0$, introduce $\tilde{\tilde{\tau}}^+$ as $\tau^+$ with $\tilde{\tilde{Y}}$ in place of $Y$ (recall Equation \eqref{def_tau}), and make similar computations to the ones to prove \cite[Lemma 5.1]{li2017general}, followed by the Markov inequality, we get
$$ \P_0\left( \sup_{t \leq \mu} |A(t \wedge \tilde{\tilde{\tau}}^+(3x_1))|>\frac{x_1}{2} \right)\leq C_1 \mu \left( \sup_{x \leq 3x_1}\sigma(x)+\sup_{x \leq 3x_1}p(x)+1 \right), $$
where $C_1$ is a finite constant.
Moreover, as $B(t)$ is a compound Poisson process, with rate $ 0< \underline{\lambda}:= \min_{x \leq 3x_1}\lambda(x)<\infty$,
there exists a positive constant $C_2$ depending on $\mu$ such that 
$$ \P_0 \left( \frac{3}{2}x_1 \leq B(\mu)\leq 2x_1 \right)\geq C_2(\mu). $$
If such a constant does not exist for a given $x_1$ because the probability $\mathfrak{I}$  has his support on $(3/2 x_1,\infty)$ for instance, we can choose a larger $x_1$ and all the properties derived until now still hold true.

As a consequence, for $x_1$ large enough, we have the following series of inequalities:
\begin{multline*}
\P_0(\tilde{\tilde{\tau}}^+(x_1)\leq \mu) 
\geq \P_0 \left( \frac{3}{2}x_1 \leq B(\mu\wedge \tilde{\tilde{\tau}}^+(3x_1))\leq 2x_1 , \sup_{t \leq \mu} |A(t \wedge \tilde{\tilde{\tau}}^+(3x_1))|\leq \frac{x_1}{2}\right) \\
=\P_0 \left( \frac{3}{2}x_1 \leq B(\mu)\leq 2x_1 , \sup_{t \leq \mu} |A(t \wedge \tilde{\tilde{\tau}}^+(3x_1))|\leq \frac{x_1}{2}\right) \\
=\P_0 \left( \frac{3}{2}x_1 \leq B(\mu)\leq 2x_1 \Big| \sup_{t \leq \mu} |A(t \wedge \tilde{\tilde{\tau}}^+(3x_1))|\leq \frac{x_1}{2}\right)\P_0 \left( \sup_{t \leq \mu} |A(t \wedge \tilde{\tau}^+(3x_1))|\leq \frac{x_1}{2}\right)  \\
=\P_0 \left( \frac{3}{2}x_1 \leq B(\mu)\leq 2x_1 \right)\P_0 \left( \sup_{t \leq \mu} |A(t \wedge \tilde{\tilde{\tau}}^+(3x_1))|\leq \frac{x_1}{2}\right)  \\
\geq C_2(\mu) \left(1- C_1 \mu \left( \sup_{x \leq 3x_1}\sigma(x)+\sup_{x \leq 3x_1}p(x)+1 \right)\right).
\end{multline*}
The first equality stems from the fact that on the event
$$ \left\{ \frac{3}{2}x_1 \leq B(\mu\wedge \tilde{\tilde{\tau}}^+(3x_1))\leq 2x_1 \right\} \bigcap  \left\{ \sup_{t \leq \mu} |A(t \wedge \tilde{\tilde{\tau}}^+(3x_1))|\leq \frac{x_1}{2} \right\}, $$
$\mu$ is smaller than $\tilde{\tilde{\tau}}^+(3x_1)$ and the last equality is a consequence of the independence of 
$N$, $Q$, $N_2$ and $B$.
We thus conclude that there exist $\mu_0<\infty$ and $C(\mu_0)>0$ such that
$$ \P_0(\tilde{\tilde{\tau}}^+(x_1)\leq \mu_0)> C(\mu_0). $$

Now, let us introduce the process $Y^{(+,x_0)}$ as follows:
\begin{itemize}
\item $Y^{(+,x)}_0= x_0$
\item $Y^{(+,x)}_.$ follows the SDE \eqref{EDS_tildetildeY} on $[0,\tilde{\tilde{\tau}}^+(3x_1)]$
\item $Y^{(+,x)}_.$ follows the SDE \eqref{EDS_tildeY} on $[\tilde{\tilde{\tau}}^+(3x_1),\infty)$.
\end{itemize}
As $\tilde{Y}$ and $\tilde{\tilde{Y}}$ are non-decreasing with respect to their initial condition, by taking the same Poisson point processes $N$, $N_1$, $N_2$ and $Q$ and the same Brownian motion $B$, when $Y_0=x_0$, we can couple the processes $Y$ and $Y^{(+,x_0)}$ to get
\begin{equation*} Y_t \geq Y^{(+,x_0)}_t \quad a.s. \quad \text{for any} \quad t \geq 0. \end{equation*}

Adding \eqref{lowerbound_explo} we deduce that
$$\inf_{x_0 \geq 0}\P_{x_0}(Y_t= \infty \text{ for } t \geq \mu_0+T_1)\geq\inf_{x_0 \geq 0}\P(Y^{(+,x_0)}_t= \infty \text{ for } t \geq \mu_0+T_1)\geq C(\mu_0) \alpha.$$
Applying a renewal argument as in the proof of point $(2)$, we get that for any $n\geq 1$,
$$\sup_{x_0 \geq 0}\P_{x_0}(\tau^+(\infty) \geq n(\mu_0+T_1))\leq(1- C(\mu_0) \alpha)^n,$$
and thus taking the limit when $n$ goes to infinity, we finally obtain
$$\inf_{x_0 \geq 0}\P_{x_0}(\tau^+(\infty) <\infty)=1.$$
The end of the proof of point $(3)$ to get a convergence in probability, via the Many-to-One formula, is the same as the end of the proof of point $(2)$, and thus leads to
$$ \lim_{t \to \infty} \mathbf{1}_{\{N_t \geq 1\}} \frac{\sum_{u \in V_t}\mathbf{1}_{\{\sup_{s\leq t}X_s^{(u)}<\infty\}}}{N_t} =0, \quad \textit{in probability}. $$
We now need to prove that the convergence is even almost sure. It is done exactly as in the proof of Proposition 2.6. iii) in \cite{marguet2020long} (this proof being a slight generalisation of the proof of Theorem 4.2. i) in \cite{BT11}), and we refer the reader to these two papers.
\end{proof}

\begin{proof}[Proof of Proposition \ref{prop_tendvers0}]
Let us begin with point (1). First notice that \eqref{rho_neg} implies that for any $x \geq 0$,
$$ g(x) \vee \E[\mathcal{P}]r(x)  \leq b x. $$
Now from Itô formula with jumps as well as from \eqref{rho_neg}, we have the following series of inequalities, when $Y_0=x_0$:
\begin{align}\label{ineg}
\E_{x_0}[e^{b(t \wedge \tau^+(x))}Y_{t \wedge \tau^+(x)}]=&x_0+ \int_0^t e^{bs} \E_{x_0}[\mathbf{1}_{s \leq\tau^+(x)}g(Y_s)]ds + \E[\mathcal{P}]\int_0^t e^{bs} \E_{x_0}[\mathbf{1}_{s \leq\tau^+(x)}r(\E_{x_0}[Y_s])]ds \nonumber \\
\leq & x_0+b\int_0^t e^{bs} \E_{x_0}[\mathbf{1}_{s \leq\tau^+(x)}Y_s]ds +b\int_0^t e^{bs} \P_{x_0}(s \leq\tau^+(x))\E_{x_0}[Y_s]ds \nonumber\\
\leq & x_0+2b\int_0^t e^{bs} \E_{x_0}[Y_s]ds. 
\end{align}
From \eqref{rho_neg}, we have that Condition \ref{SNinfty} is satisfied, and thus $\lim_{x \to \infty} \tau^+(x)=\infty$ a.s. Hence from Fatou's Lemma we get that
\begin{align*}
e^{bt}\E_{x_0}[Y_{t}]\leq x_0+2b\int_0^t e^{bs} \E_{x_0}[Y_s]ds,
\end{align*}
and we conclude by Grömwall's Lemma that 
$$ e^{bt}\E_{x_0}[Y_{t}]\leq x_0 e^{2bt}<\infty. $$
Applying inequality \eqref{rho_neg} in the first line of \eqref{ineg} as well as Fatou's Lemma we obtain
\begin{align*}
e^{bt}\E_{x_0}[Y_{t}]\leq x_0+ \liminf_{x \to \infty} \left(b\int_0^t e^{bs} \P_{x_0}(s \leq\tau^+(x))\E_{x_0}[Y_s]ds - \int_0^t \P_{x_0}(s \leq\tau^+(x))(\alpha\E_{x_0}[Y_s]\wedge \eta)e^{bs}ds \right. \\
\left. + \int_0^t e^{bs} \E_{x_0}[\mathbf{1}_{s \leq\tau^+(x)}(g(Y_s)-g(\E_{x_0}[Y_s]))]ds \right).
\end{align*}
As for any $s\geq 0$, $\mathbf{1}_{s \leq\tau^+(x)}$ is smaller than one and non-decreasing with $x$, $g(Y_s)\leq b Y_s$ and $\E_{x_0}[Y_s]<\infty$ we get by the Monotone Convergence Theorem,
\begin{align*}
e^{bt}\E_{x_0}[Y_{t}]\leq x_0+ b\int_0^t e^{bs} \E_{x_0}[Y_s]ds - \int_0^t (\alpha\E_{x_0}[Y_s]\wedge \eta)e^{bs}ds
 + \int_0^t e^{bs} \E_{x_0}[g(Y_s)-g(\E_{x_0}[Y_s])]ds.
\end{align*}
As $g$ is concave, Jensen inequality entails
\begin{equation} \label{maj_Ey} e^{bt}\E_{x_0}[Y_{t}]\leq x_0+ b\int_0^t e^{bs} \E_{x_0}[Y_s]ds - \int_0^t (\alpha\E_{x_0}[Y_s]\wedge \eta)e^{bs}ds. \end{equation}
Let us introduce:
$$ \mathcal{T}_1 := \inf \left\{ t \geq 0, \E[Y_t]< \frac{\eta}{\alpha} \right\} $$
and
$$ \mathcal{T}_2 := \inf \left\{ t \geq \mathcal{T}_1, \E[Y_t]\geq \frac{\eta}{\alpha} \right\} $$
(if $\mathcal{T}_1=0$ we can choose $\eta<\alpha x_0$, which will make $\mathcal{T}_1>0$ and not modify the later computations).
We will prove that $\mathcal{T}_1<\infty = \mathcal{T}_2$.
To this aim, let us do a reductio ad absurdum by assuming that $\mathcal{T}_1=\infty$. In this case, \eqref{maj_Ey} writes for any $t\geq 0$
\begin{align} \label{maj_EYt} e^{bt}\E_{x_0}[Y_{t}]\leq x_0+ b\int_0^t e^{bs} \E_{x_0}[Y_s]ds - \int_0^t  \eta e^{bs}ds, \end{align}
and if we introduce the function 
$$ \mathcal{G}(t) = e^{-bt}\int_0^t e^{bs}\E_{x_0}[Y_s]ds, $$
we obtain
\begin{align*} \mathcal{G}'(t) &= -be^{-bt}\int_0^t e^{bs}\E_{x_0}[Y_s]ds + \E_{x_0}[Y_t] \\
& \leq -be^{-bt}\int_0^t e^{bs}\E_{x_0}[Y_s]ds + e^{-bt}\left( x_0+ b\int_0^t e^{bs} \E_{x_0}[Y_s]ds - \int_0^t  \eta e^{bs}ds \right)\\
& =  e^{-bt}\left( y_0 - \int_0^t  \eta e^{bs}ds \right)\\
&= \left(y_0+ \frac{\eta}{b}\right)e^{-bt}- \frac{\eta}{b}, \end{align*}
where we applied \eqref{maj_EYt} to obtain the inequality.
This implies that for every $t\geq 0$,
$$ \mathcal{G}(t) \leq \left(x_0+ \frac{\eta}{b}\right)\frac{1-e^{-bt}}{b}- \frac{\eta}{b}t $$
which is not possible as $Y_t$ is non-negative.
There is thus a contradiction and we deduce that 
$\mathcal{T}_1<\infty$.
Notice now that if $\mathcal{T}_2<\infty$, then 
$$ e^{b \mathcal{T}_2}\E_{x_0}[Y_{\mathcal{T}_2}]\leq e^{b \mathcal{T}_1}
\E_{x_0}[Y_{\mathcal{T}_1}]+ (b-\alpha)\int_{\mathcal{T}_1}^{\mathcal{T}_2}e^{bs}\E_{x_0}[Y_s]ds, $$
which entails by an application of Grömwall's Lemma
$$ \E_{x_0}[Y_{\mathcal{T}_2}]\leq e^{-\alpha (\mathcal{T}_2- \mathcal{T}_1)}\E_{x_0}[Y_{\mathcal{T}_1}]\leq e^{-\alpha (\mathcal{T}_2- \mathcal{T}_1)}\frac{\eta}{\alpha}. $$
This contradicts the definition of $\mathcal{T}_2$, which proves that 
$\mathcal{T}_2=\infty$.
Notice that all the previous calculations stay true if we take a smaller $\eta$. This thus proves the following property:
\begin{equation*}
\lim_{t \to \infty} \E_{x_0}[Y_t]=0.
\end{equation*}
The end of the proof of point (1) follows the scheme we already used: Markov inequality to show that for any $\eps>0$,
$$\lim_{t \to \infty}\P_{x_0}(Y_t>\eps) = 0,$$
Many-to-One formula, and convergence of $N_te^{-(b-d)t}$ on the survival event of the cell population.\\

Let us now prove point (2). The proof is very similar to the one of Theorem \ref{cor_esperance}.(3). Equation \eqref{lowerbound_explo} holds for $Y$. Hence we just have to prove that for any $x_1<\infty$ satisfying \eqref{lowerbound_explo}, there exists $T_0<\infty$ such that
\begin{equation} \label{needed_ineq} \P_{x_0}(\tau^+(x_1)<T_0)>0. \end{equation}
Indeed in this case it will imply that
$$ \P_{x_0}(\tau^+(\infty)\leq T_0+ T_1)>0, $$
and thus 
$$ r(\E_{x_0}[Y_{t}])= r(\infty)>0, \quad \forall t \geq T_0+T_1. $$ 
Indeed, as \ref{LNinfty} implies \eqref{cond_rest_infty} the expectation of $Y_t$ stays infinity after the time $T_0+T_1$ (see point (1) of Theorem \ref{theo_CDI}).
Once we have this property the end of the proof is the same.

Let us thus introduce 
$$T:=\tau^-(0) \wedge \tau^+(2x_1).$$
Then following the proof of Lemma 5.1 in \cite{li2017general} we get for $\delta<\infty$,
$$ \P_{x_0}\left( \sup_{t \leq \delta} \left|\int_0^{t\wedge T}\sqrt{2\sigma^2 (Y_s)}dB_s+\int_0^{t\wedge T}\int_0^{p(Y_{s-})}\int_{\mathbb{R}_+}z\widetilde{Q}(ds,dx,dz) \right|>\sqrt{\delta} \right)\leq C(x_1)\sqrt{\delta}. $$
Moreover, 
$$ \P_{x_0}\left(\int_0^\delta\int_0^{2b}\int_{0}^1(\theta-1)\tilde{Y}_{s-}N(ds,dx,d\theta)=0 \right)= e^{-2b\delta}.$$
Now if $\delta$ is small enough to satisfy $r(x_0-\sqrt{\delta})>r(x_0)/2>0$ (such a $\delta$ exists as $r$ is a continuous function), then 
$$ \P\left( \int_0^\delta\int_0^{r(x_0)/2}\int_{\mathbb{R}_+}\mathfrak{p}N_1(ds,dx,d\mathfrak{p})>2x_1 \right) = C(x_0,x_1)>0. $$
Combining these three computations as in the proof of Proposition \ref{cor_esperance}.(3) we get that \eqref{needed_ineq} holds true, which ends the proof.
\end{proof}

\begin{proof}[Proof of Proposition \ref{theo_CDI}]

Recall the definition of $G_a^{(t)}(x_0,x)$ in \eqref{eq:Ga1} and that $Z^{(a)}_t$ defined in \eqref{def_Za} is a $\mathcal{F}_t$-martingale, and notice that under the assumptions of Proposition \ref{theo_CDI}, as $r(\cdot)\equiv 0$, $G_a^{(t)}(x_0,x)$ does not depend on $(x_0,t)$. We will thus drop these dependencies in the proof.
For the sake of readability we will use the following Conditions \eqref{cond_rest_infty_bis} and \eqref{cond_comdown_infty_bis}, equivalent to the Conditions \eqref{cond_rest_infty} and \eqref{cond_comdown_infty}, respectively:
\begin{equation}\label{cond_rest_infty_bis}
\exists a \in \mathcal{A} \quad \text{and a non-negative function } f, \quad
G_a(x) \geq f(x)+ o(\ln x), \quad (x \to \infty),
\end{equation} 
\begin{equation}\label{cond_comdown_infty_bis}
\exists 0<a<1, \quad \eta>0, \quad G_a(x) \geq  \ln x(\ln \ln x)^{1+\eta}, \quad (x \to \infty).
\end{equation}
This can be seen by noticing that the expressions on the left hand side of \eqref{cond_rest_infty} and \eqref{cond_comdown_infty} equal 
$$ \frac{G_a(x)}{a-1}+2b \frac{1-\E[\Theta^{1-a}]}{1-a}. $$

Here again we begin with the study of the auxiliary process $Y$. Recall the definition of $\tau^-$ in \eqref{def_tau}. Then we have the two following properties:
\begin{enumerate}
\item If there exists $a \in \mathcal{A}$ and a non-negative function $f$ such that
\eqref{cond_rest_infty_bis} holds,
then for any $\mathfrak{b},\mathfrak{d}>0$, 
\begin{equation} \label{cas1_lemCDI}
\lim_{x \to \infty} \P_x(\tau^-(\mathfrak{b})<\mathfrak{d})=0.
\end{equation}
\item If there exist $0<a<1$, $\eta>0$, such that
\eqref{cond_comdown_infty_bis} holds, 
then for $\delta$ small enough and $\mathfrak{b}$ large enough,
\begin{equation}\label{quantif_CDI}
\inf_{x >1}\P_x (\tau^-(\mathfrak{b})\leq l(\mathfrak{b},\delta,\eta))\geq e^{-8\mathfrak{b}^{-\delta(1-a)}},
\end{equation}
where $l(\mathfrak{b},\delta,\eta)$ is a function of $\mathfrak{b}$, $\delta$ and $\eta$, finite for $\mathfrak{b}$ large enough and $\delta$ small enough, which will be defined in \eqref{def_l}
\end{enumerate}

Let us begin by proving \eqref{cas1_lemCDI}. The beginning of the proof follows ideas of the proof of \cite[Theorem 2.13]{li2017general} but needs to be adapted as our bound is sharper and the process under consideration exhibits negative jumps unlike the processes considered in \cite{li2017general}.
Let us take positive constants $\mathfrak{d}$, $\alpha<(a-1)/4\mathfrak{d}$, and $\mathfrak{b}$ such that $G_a(x)\geq -\alpha \ln x$ holds for $x>\mathfrak{b}$. For any $0 < \eps<\mathfrak{b}^{-1}$, Equation (5.17) of \cite{li2017general} holds.
Then following the computations of \cite{li2017general} with $-\alpha \ln x$ in place of $-(\ln x)^r$ with $0<r<1$, we get for $\eps^{-2^n}\leq x < \eps^{-2^{n+1}}$:
$$ \P_x\left( \tau^-(\mathfrak{b})<\tau^+(\eps^{-2^{n+1}}) \wedge \mathfrak{d} \right)\leq \mathfrak{b}^{a-1}\eps^{(a-1-2\alpha \mathfrak{d})2^n} \leq \mathfrak{b}^{a-1}\eps^{(a-1)2^{n-1}}, $$
where in the last inequality we used that $\alpha \leq (a-1)/4\mathfrak{d}$. The end of the proof of \eqref{cas1_lemCDI} follows the end of the proof of point $(i)$ of \cite[Theorem 2.13]{li2017general}.\\

Let us now prove \eqref{quantif_CDI}. Here again we essentially follow the proof of point $(i)$ of \cite[Theorem 2.13]{li2017general}, but the process $Y$ may experience negative jumps unlike the processes considered in \cite{li2017general}, we take sharper bounds, and we need to adapt the proof. We just list the modifications to be made to cover our case. Instead of taking $t(x)= (1+\delta)^r (\ln x)^{1-r}$ (with $0<r<1$, $\delta>0$), we choose
$$ t(x):= \left( \ln \ln x^{1/(1+\delta)}\right)^{-(1+\eta)}, \quad \forall x>1, $$
where $\eta$ is chosen as in \eqref{cond_comdown_infty}. 
Let $0<\theta<1$. By taking $\theta x$ instead of $x$ in the computations at the end of \cite[p.18]{li2017general}, we get in place of their inequality $(5.18)$:
\begin{align*} \P_{\theta x}\left(\tau^-(x^{(1+\delta)^{-1}})>\tau^+(x^{(1+\delta)})\right)
&\leq \mathbf{1}_{\{ \theta x > x^{(1+\delta)^{-1}} \}}\theta^{1-a} x^{-\delta(1-a)}\\
&\leq \mathbf{1}_{\{ \theta x > x^{(1+\delta)^{-1}} \}}x^{-\delta(1-a)}, \end{align*}
as $a<1$.
The two first inequalities of \cite[p.33]{li2017general} are modified as follows:
\begin{align*}
\theta^{1-a}x^{1-a} \geq & \E_{\theta x} \left[ Y^{1-a}_{\tau^-(x^{(1+\delta)^{-1}})}\exp\left(\int_0^{\tau^-(x^{(1+\delta)^{-1}})}G_a(Y_s)ds\right)\mathbf{1}_{\{ t(x)<\tau^-(x^{(1+\delta)^{-1}})<\tau^+(x^{(1+\delta)}) \}} \right]\\
\geq & x^{\frac{1-a}{1+\delta}}\E[\Theta^{1-a}] \E_{\theta x} \left[ \exp\left(\int_0^{t(x)}\ln Y_s (\ln \ln Y_s)^{1+\eta}ds\right)\mathbf{1}_{\{ t(x)<\tau^-(x^{(1+\delta)^{-1}})<\tau^+(x^{(1+\delta)}) \}} \right],
\end{align*}
as a negative jump may occur at time $\tau^-(x^{(1+\delta)^{-1}})$, which leads to
$$ \P_{\theta x}(t(x)<\tau^-(x^{(1+\delta)^{-1}})<\tau^+(x^{(1+\delta)}))\leq \mathbf{1}_{\{ \theta x > x^{(1+\delta)^{-1}} \}}x^{-(1+a\delta)/(1+\delta)}. $$
If $\delta$ is small enough and $x$ large enough, we thus still obtain the inequality:
$$ \P_x(\tau^-(x^{(1+\delta)^{-1}}))\leq 2 \mathbf{1}_{\{ \theta x > x^{(1+\delta)^{-1}} \}} x^{(\delta-(1+a\delta))/(1+\delta)}\leq 2  x^{-\delta(1-a)}. $$
This proves that in the computations of the end of \cite[p.33]{li2017general}, we may replace the terms of the form 
$$ \P_{\mathfrak{b}^{(1+\delta)^n}}(.) $$
by terms of the form
$$ \P_{\tau^-(\mathfrak{b}^{(1+\delta)^n})}(.) $$
without modifying the bounds.
To end the proof, we still have to check that with our choice for the function $t$, 
$$ \sum_{n=1}^\infty t(\mathfrak{b}^{(1+\delta)^n})<\infty, $$
which is true. 
More precisely we have
\begin{equation} \label{def_l} \sum_{n=1}^\infty t(\mathfrak{b}^{(1+\delta)^n})= \sum_{n=1}^\infty \left((n-1)\ln (1+\delta)+ \ln \ln \mathfrak{b} \right)^{-(1+\eta)}=:l(\mathfrak{b},\delta,\eta). \end{equation}
We may thus obtain the following inequality,
\begin{equation*}
\inf_{x >1}\P_x (\tau^-(\mathfrak{b})\leq l(\mathfrak{b},\delta,\eta))\geq e^{-8\mathfrak{b}^{-\delta(1-a)}}
\end{equation*}
which ends the proof of \eqref{quantif_CDI}.\\

We may now prove Proposition \ref{theo_CDI}.
Let us begin with point $(1)$. By the Many-to-One formula \eqref{MtO}, and the definition of $\tau^-(\mathfrak{b})$ we have
$$ e^{-(b-d)t}\E_x \left[\sum_{u \in V_t}\mathbf{1}_{\{X_t^{u}\leq \mathfrak{b}\}}\right]= 
\P_x \left(Y_t\leq \mathfrak{b}\right)\leq \P_x (\tau^-(\mathfrak{b}) \leq t)\to 0, \quad (x \to \infty), $$
where we used \eqref{cas1_lemCDI} to obtain the last limit. We conclude as before using the convergence of $N_te^{-(b-d) t}$ on the event of the cell population survival.\\

Let us now consider point $(2)$. Let $t,\eps>0$, and $\delta$ such that \eqref{quantif_CDI} holds for $\mathfrak{b}$ large enough. By definition of $l(\mathfrak{b},\delta,\eta)$ in \eqref{def_l},
$$ \lim_{\mathfrak{b} \to \infty} l(\mathfrak{b},\delta,\eta)=0. $$
We also have
$$ \lim_{\mathfrak{b} \to \infty} e^{-8\mathfrak{b}^{-\delta(1-a)}}=1. $$
Hence there exists $\mathfrak{b}_0(t,\delta,\eta,\eps)$ such that if $\mathfrak{b}\geq \mathfrak{b}_0(t,\delta,\eta,\eps)$,
\begin{equation*}
\inf_{x >1}\P_x (\tau^-(\mathfrak{b})\leq t)\geq 1-\eps.
\end{equation*}

Now, notice that \eqref{cond_comdown_infty_bis} implies that Conditions of \cite[Theorem 2.8(i)]{li2017general} (or its extension \cite[Theorem 4.1.i)]{companion}) are satisfied and thus for any $x\geq 0$, 
\begin{equation} \label{esp_nulle} \E_x[e^{-\tau^+(\infty)}]=0, \end{equation}
where we recall that
$$ \tau^+(\infty):= \lim_{y \to \infty}\tau^+(y). $$
As the process $Y$ admits negative jumps, we have for any $n \in \N$ the inequality
$$ \E \left[ \E_{Y_{\tau^-(\mathfrak{b})}}[e^{-\tau^+(\mathfrak{b}^n)}] \right]\leq  \E_{\mathfrak{b}}[e^{-\tau^+(\mathfrak{b}^n)}] + \int_0^1 \E_{\theta \mathfrak{b}}[e^{-\tau^+(\mathfrak{b}^n)}]\kappa(d\theta). $$
For any $\theta \in [0,1]$, $\E_{\theta \mathfrak{b}}[e^{-\tau^+(\mathfrak{b}^n)}]$ is non-increasing with $n$, and according to \eqref{esp_nulle}
$$ \E_{\theta \mathfrak{b}}[e^{-\tau^+(\mathfrak{b}^n)}] \to 0, \quad n\to \infty. $$
By the Monotone Convergence Theorem, we thus deduce that 
$$ \E \left[ \E_{Y_{\tau^-(\mathfrak{b})}}[e^{-\tau^+(\mathfrak{b}^n)}] \right]\to 0, \quad n\to \infty. $$
Hence, for any $\eps,t>0$, $\mathfrak{b}>1$, there exists $n(\mathfrak{b},t) \in \N$ such that for any $n \geq n(\mathfrak{b})$,
$$ \E \left[ \E_{Y_{\tau^-(\mathfrak{b})}}[e^{-\tau^+(\mathfrak{b}^n)}] \right]\leq \eps e^{-t}, $$
which implies, by the Markov inequality, 
\begin{align*}\E \left[ \P_{Y_{\tau^-(\mathfrak{b})}}(\tau^+(\mathfrak{b}^n)\leq t) \right] & =
\E \left[ \P_{Y_{\tau^-(\mathfrak{b})}}(e^{-\tau^+(\mathfrak{b}^n)})\geq e^{-t} \right]
\leq  \frac{\E \left[ \E_{Y_{\tau^-(\mathfrak{b})}}[e^{-\tau^+(\mathfrak{b}^n)}] \right]}{e^{-t}}\leq \eps . \end{align*}
To conclude, for any $t,\eps>0$ and $\delta$ small enough, if we choose $\mathfrak{b}\geq \mathfrak{b}_0(t,\delta,\eta,\eps)$ and $n \geq n(\mathfrak{b})$ we obtain thanks to the strong Markov inequality, and if we recall that $(\mathcal{F}_t, t \geq 0)$ is the canonical filtration associated to $Y$, for any $x>1$,
\begin{align*}
\P_x(Y_t>\mathfrak{b}^n) & \leq  \P_x(\tau^-(\mathfrak{b})>t)+ \P_x(Y_t>\mathfrak{b}^n, \tau^-(\mathfrak{b})\leq t)\\
& \leq  \P_x(\tau^-(\mathfrak{b})>t)+ \E \left[\P_x(Y_t>\mathfrak{b}^n, \tau^-(\mathfrak{b})\leq t|\mathcal{F}_{\tau^-(\mathfrak{b})})\right]\\
& \leq  \P_x(\tau^-(\mathfrak{b})>t)+\E \left[ \P_{Y_{\tau^-(\mathfrak{b})}}(\tau^+(\mathfrak{b}^n)\leq t) \right] \leq 2 \eps.
\end{align*}
We conclude as before using the Many-to-One formula \eqref{MtO} and the convergence of $N_te^{-(b-d) t}$ on the event of the cell population survival.
\end{proof}

\section*{Acknowledgments}
The author thanks Aline Marguet for her comments and suggestions.
This work was partially funded by the Chair "Mod\'elisation Math\'ematique et Biodiversit\'e" of VEOLIA-Ecole Polytechnique-MNHN-F.X. This work has also been partially supported by the LabEx PERSYVAL-Lab (ANR-11-LABX-0025-01) funded by the French program Investissement d’avenir.

\bibliographystyle{abbrv}
\bibliography{biblio}

\end{document}